\documentclass{article}
\usepackage{amsmath, amsfonts, amsthm, amssymb}
\usepackage{enumerate}

\theoremstyle{plain}
\newtheorem{proposition}{Proposition}[section]
\newtheorem{theorem}[proposition]{Theorem} 
\newtheorem{lemma}[proposition]{Lemma}

\newtheorem{corollary}[proposition]{Corollary}

\theoremstyle{definition}

\newtheorem{remark}[proposition]{Remark}

\def\sdisc{{\mathrm{sDisc}}}
\def\disc{{\mathrm{Disc}}}
\def\diag{{\mathcal{D}}}
\def\matrices{{\mathcal{L}}}
\def\mr{{\mathbb{R}}}
\def\mn{{\mathbb{N}}}
\def\mc{{\mathbb{C}}}
\def\mf{{\mathbb{F}}}

\def\rsm{{\mathcal{M}}}
\def\dsm{{\mathcal{E}}}
\def\tzrsm{{\mathcal{N}}}
\def\cov{{\mathcal{T}}}
\def\mt{{\mathbb{T}}}

\begin{document}
\title{Invariant theoretic characterization of subdiscriminants of  matrices} 
\author{M. Domokos\thanks{Partially supported by OTKA  NK81203 and K101515.}
\\ {\small R\'enyi Institute of Mathematics, Hungarian Academy of 
Sciences,} 
\\ {\small 1053 Budapest, Re\'altanoda utca 13-15., Hungary} \\ 
{\small E-mail: domokos.matyas@renyi.mta.hu } 
}
\date{}
\maketitle 
\begin{abstract} An invariant theoretic characterization of subdiscriminants of  matrices is given. 
The structure as a module over the special orthogonal group of the minimal degree non-zero homogeneous component of the vanishing ideal of the variety of real symmetric matrices with a bounded number of different eigenvalues is investigated. 
These results are applied  to the study of sum of squares presentations of subdisciminants of real symmetric matrices. 


\end{abstract}


\section{Introduction}\label{sec:intro}

Let $n$ be a positive integer and $k\in\{0,1,\dots,n-1\}$. 
The {\it $k$-subdiscriminant} of a degree $n$ monic polynomial $p=\prod_{i=1}^n(x-\lambda_i)\in\mc[x]$ with complex roots $\lambda_1,\dots,\lambda_n$ is 
\[\sdisc_k(p):=\sum_{1\le i_1<\dots<i_{n-k}\le n}\delta(\lambda_{i_1},\dots,\lambda_{i_{n-k}})^2\] 
where 
\[\delta(x_1,\dots,x_{n-k}):=\prod_{1\le i<j\le n-k}(x_i-x_j).\]  
(For $k=n-1$ we have $\sdisc_{n-1}=n$.) 
It can be written as a polynomial function (with integer coefficients) of  the coefficients of $p$. Moreover, $p$ has exactly $n-k$ distinct roots in $\mc$ if and only if  
$\sdisc_0(p)=\dots=\sdisc_{k-1}(p)=0$ and $\sdisc_k(p)\neq 0$. 
The relevance of subdiscriminants  for counting real roots of polynomials $p\in\mr[x]$ is explained in Chapter 4 of \cite{roy}. 

Given an $n\times n$ matrix $A$ (say with complex entries) its $k$-subdiscriminant ($k\in\{0,1,\dots,n-1\}$) is defined 
as $\sdisc_k(A):=\sdisc_k(p_A)$, where $p_A$ is the characteristic polynomial of $A$. 
Obviously $\sdisc_k(A)$ is a homogeneous polynomial function (with integer coefficients)  in the entries of $A$ of degree $(n-k)(n-k-1)$. 
The matrix  $A$ has exactly $n-k$ distinct complex eigenvalues if and only if  
$\sdisc_0(A)=\dots=\sdisc_{k-1}(A)=0$ and $\sdisc_k(A)\neq 0$. 
In the special case $k=0$ we recover the  discriminant $\disc(A):=\sdisc_0(A)$. 

Up to non-zero scalar multiples the discriminant is the only degree $n(n-1)$ homogeneous polynomial function on the space of matrices having both of the following two properties: (i) it  vanishes on all degenerate matrices (i.e. matrices with a multiple eigenvalue); (ii) it is invariant under the action of the general linear group by conjugation. 
This statement is well known (a version for real symmetric matrices is the starting point of \cite{lax:1998}). 
In the first half of the present note we generalize it and show that a similar invariant theoretic characterization of the $k$-subdiscriminant of matrices is valid for all $k$. Our Theorem~\ref{thm:GL_n-inv} asserts that up to non-zero scalar multiples $\sdisc_k$ is the only homogeneous polynomial $GL_n$-invariant function of degree 
$(n-k)(n-k-1)$ on the space of $n\times n$  matrices  that  vanishes on all matrices with at most $n-k-1$ different eigenvalues, and there is no such $GL_n$-invariant 
polynomial of smaller degree. The proof of this result depends on the Kleitman-Lov\'asz Theorem (cf. \cite{lovasz}) giving generators of the vanishing ideal in $\mc^n$ 
of the subspace arrangement consisting of the points with at most $n-k-1$ distinct coordinates.

Subdiscriminants of real symmetric matrices are particularly interesting:  all eigenvalues of a real symmetric matrix are real, therefore 
$\sdisc_k(A)=0$ for a real symmetric $n\times n$ matrix $A$ if and only if $A$ has at most $n-k-1$ different eigenvalues.  
Theorem~\ref{thm:GL_n-inv} has a variant Theorem~\ref{thm:SO_n-inv} for real symmetric matrices: up to non-zero scalar multiples $\sdisc_k$ is the only $SO_n$-invariant homogeneous polynomial function of degree $(n-k)(n-k-1)$ on the space $\rsm$ of $n\times n$ real symmetric matrices that vanishes on the set $\dsm_k$ of real symmetric matrices with at most $n-k-1$ distinct eigenvalues, and there is no such $SO_n$-invariant of smaller degree. We also show that the minimal degree of a non-zero polynomial function on the space of real symmetric matrices that vanishes on $\dsm_k$ is $(n-k)(n-k-1)/2$, see Corollary~\ref{cor:I(E_k)}. 

Note that  the subdiscriminants are non-negative forms on the space $\rsm$ of real symmetric matrices. 
We apply Theorem~\ref{thm:SO_n-inv}  to the study of sum of squares presentations of the subdiscriminants. 
View $\sdisc_k$ as an element of the coordinate ring 
$\mr[\rsm]$ of $\rsm$. 
In the special case $k=0$ the fact that the discriminant $\disc$  can be written as a sum of squares in the $n(n+1)/2$-variable polynomial ring 
$\mr[\rsm]$  goes back to Kummer and Borchardt, and was rediscovered and  refined by several authors, see \cite{domokos} for references 
(see also \cite{gorodski} for a generalization and  \cite{sanyal-sturmfels-vinzant} for a recent application of sum of squares presentations of the discriminant of symmetric matrices).  
It was shown by Roy (see Theorem 4.48 in  \cite{roy}) that the $k$-subdiscriminant is  a sum of squares for all $k=0,1,\dots,n-2$ as well, in fact she presented $\sdisc_k$ explicitly  as a  sum of squares (a generalization of this  in the context of semisimple symmetric spaces was communicated to me by Ra{\"\i}s \cite{rais}). This motivates the following definition: 
\[\mu_k(n):=\min\{r\in \mn\mid \exists f_1,\dots,f_r\in \mr[\rsm]: \sdisc_k=\sum_{i=1}^r f_i^2\}\]
In the special case $k=0$ the number $\mu(n):=\mu_0(n)$ is investigated in \cite{domokos} (building on the ideas of  \cite{lax:1998}), where it is shown that for $n\ge 3$, the number $\mu(n)$ is bounded by the dimension of the space of $n$-variable spherical harmonics of degree $n$ (an irreducible representation of $SO_n$).  It turns out that  the approach of \cite{lax:1998}, \cite{domokos} can be extended  for the $k$-subdiscriminant as well. We shall locate an irreducible $SO_n$-module direct summand in the degree $(n-k)(n-k-1)/2$ homogeneous component of the vanishing ideal of the 
subvariety $\dsm_k$ of $\rsm$, see Theorem~\ref{thm:main}.  As a corollary of the characterization of $\sdisc_k$ given in Theorem~\ref{thm:SO_n-inv} we conclude that $\mu_k(n)$ is bounded from above by the dimension of the above mentioned irreducible $SO_n$-module (Corollary~\ref{cor:bound}). This yields a significant improvement of the bound 
on $\mu_k(n)$ provided by the explicit sum of squares presentation of $\sdisc_k$ given in \cite{roy}. 

\begin{center} {\it Acknowledgement. } \end{center} 
The author is grateful to Marie-Francoise Roy for her suggestion to study the subdiscriminants along the lines of \cite{lax:1998}, \cite{domokos}, and for inspiring discussions during a visit of the author to Rennes.


\section{Some results on symmetric polynomials}\label{sec:invar}

Let $\mf$ be a field of characteristic zero, and 
denote by $\diag_k$ the subset of $\diag:=\mf^n$ consisting of points with at most $n-k-1$ different coordinates. It is an $(n-k-1)$-dimensional subspace arrangement. 
Write $I(\diag_k)$ for the vanishing ideal of $\diag_k$ in the coordinate ring $\mf[\diag]=\mf[x_1,\dots,x_n]$ (an $n$-variable polynomial ring). 
The symmetric group $S_n$ acts on $\diag$ by permuting coordinates. This induces a left  action of $S_n$ on the coordinate ring $\mf[\diag]$ of $\mf^n$ given by 
$\pi\cdot f(x_1,\dots,x_n):=f(x_{\pi(1)},\dots,x_{\pi(n)})$ for $\pi\in S_n$ and $f\in\mf[\diag]$. 
The corresponding subalgebra of invariants is 
\[\mf[\diag]^{S_n}:=\{f\in\mf[x_1,\dots,x_n]\mid \pi\cdot f=f \quad \forall \pi\in S_n\}\] 
 (the algebra of $n$-variable symmetric polynomials). 
The subset $\diag_k$ is $S_n$-stable, hence $I(\diag_k)$ is an $S_n$-submodule in $\mf[\diag]$. 
Set 
\[I(\diag_k)^{S_n}:=I(\diag_k)\cap\mf[x_1,\dots,x_n]^{S_n}\]  
and denote by $I(\diag_k)^{S_n}_d$ the degree $d$ homogeneous component of $I(\diag_k)$. 
We put 
\[\Delta_{n,k}:=\sum_{1\le i_1<\dots<i_{n-k}\le n}\delta(x_{i_1},\dots,x_{i_{n-k}})^2\] 

\begin{lemma}\label{lemma:S_n} 
The space $I(\diag_k)^{S_n}_{(n-k)(n-k-1)}$ is spanned by $\Delta_{n,k}$, and
there are no $S_n$-invariants in $I(\diag_k)$ of degree less than 
$(n-k)(n-k-1)$. 
\end{lemma} 

\begin{proof} By the Kleitman-Lov\'asz Theorem \cite{lovasz} the ideal $I(\diag_k)$ of $\mf[x_1,\dots,x_n]$ is generated by the polynomials 
$\delta(x_{i_1},\dots,x_{i_{n-k}})$, where $1\le i_1<\dots <i_{n-k}\le n$. 
The Reynolds operator $\tau:\mf[x_1,\dots,x_n]\to\mf[x_1,\dots,x_n]^{S_n}$ is given by 
$\tau(f):=\frac{1}{n!}\sum_{g\in S_n}g\cdot f$. It is a projection onto the subspace of symmetric polynomials, preserving the degree and mapping $I(\diag_k)_d$ onto $I(\diag_k)^{S_n}_d$. 
The $S_n$-orbits of the polynomials $\delta(x_1,\ldots,x_{n-k})m$ where $m$ is a monomial in $x_1,\dots,x_n$ span $I(\diag_k)$, therefore $I(\diag_k)^{S_n}$ is spanned by $\tau(\delta(x_1,\dots,x_{n-k})m)$, where $m$ ranges over the set of monomials. Identify $S_{n-k}$ with the subgroup of $S_n$ consisting of permutations fixing $n-k+1,\dots,n$, and write $S_n/S_{n-k}$ for a system of left $S_{n-k}$-coset representatives in $S_n$. 
Observe that for $g\in S_{n-k}$ we have $g\cdot \delta(x_1,\dots,x_{n-k})=\mathrm{sign}(g)\delta(x_1,\dots,x_{n-k})$. 
Therefore 
\begin{eqnarray*}\tau(x_1^{\alpha_1}\dots x_n^{\alpha_n}&\delta(x_1,\dots,x_{n-k}))
=\frac{1}{n!}\sum_{h\in S_n/S_{n-k}}h\cdot \left(\sum_{g\in S_{n-k}}g\cdot (x^{\alpha}\delta)\right)\\
&=\frac{1}{n!}\sum_{h\in S_n/S_{n-k}}h\cdot \left( \sum_{g\in S_{n-k}}x_{g(1)}^{\alpha_1}\dots x_{g(n)}^{\alpha_n}
\mathrm{sign}(g)\delta \right)
\end{eqnarray*} 
Observe that 
$\sum_{g\in S_{n-k}}\mathrm{sign}(g)x_{g(1)}^{\alpha_1}\dots x_{g(n)}^{\alpha_n}=0$ unless 
$\alpha_1,\dots,\alpha_{n-k}$ are all distinct. 
It follows that  if $\alpha_1+\cdots+\alpha_{n-k}<0+1+\cdots+(n-k-1)$, then $\tau(x^{\alpha} \delta)=0$, whereas 
if $\{\alpha_1,\dots,\alpha_{n-k}\}=\{0,1,\dots,n-k-1\}$ and $\alpha_{n-k+1}=\dots=\alpha_n=0$, then 
$\sum_{g\in S_{n-k}}\mathrm{sign}(g)x_{g(1)}^{\alpha_1}\dots x_{g(n)}^{\alpha_n}=\pm \delta(x_1,\dots,x_{n-k})$, 
implying 
\begin{eqnarray*}
\tau(x^{\alpha}\delta)&=\pm\frac{1}{n!}\sum_{h\in S_n/S_{n-k}}h\cdot \left(\delta(x_1,\dots,x_{n-k})^2 \right)
=\pm\frac{1}{n!}\Delta_{n,k}\end{eqnarray*} 
\end{proof} 

\begin{remark}\label{remark:subspacearrangement} (i) The assumption that the characteristic of $\mf$ is zero was necessary for the use of the Reynolds operator in the above proof. 

(ii) Vanishing ideals of subspace arrangements are intensively studied and there are several open questions about them, see for example \cite{bjoerner-peeva-sidman}.  For the particular case of $\diag_k$, in addition to the Kleitman-Lov\'asz Theorem \cite{lovasz} cited above, it was also shown in \cite{deloera} that the above generators constitute a universal Gr\"obner basis; see also \cite{domokos:1999} and \cite{sidman}. 
\end{remark} 


\section{Conjugation invariants}\label{sec:conj}

Assume in this section that our base field $\mf$ is algebraically closed and has characteristic zero. 
The space $\matrices:=\mf^{n\times n}$ of $n\times n$ matrices over $\mf$ contains the subset 
$\matrices_k$ consisting of matrices with at most $n-k-1$ distinct eigenvalues for $k=0,1,2,\dots,n-2$.  
It is well known that $\matrices_k$ is Zariski closed in $\matrices$; we shall denote by $I(\matrices_k)$ the vanishing ideal of 
$\matrices_k$ in the coordinate ring $\mf[\matrices]$ of $\matrices$, an $n^2$-variable polynomial ring over $\mf$. 
The general linear group $GL_n:=GL_n(\mf)$ acts by conjugation on $\matrices$, and we shall write 
$\mf[\matrices]^{GL_n}$ for the corresponding subalgebra of invariants. Clearly $\matrices_k$ is a $GL_n$-stable subset of $\matrices$, hence $I(\matrices_k)$ is a $GL_n$-submodule in 
$\mf[\matrices]$ (endowed with the action induced by the action of $GL_n$ on $\matrices$ in the standard way). We shall denote by $I(\matrices_k)^{GL_n}$ the 
subspace of $GL_n$-invariants in $I(\matrices_k)$. Obviously $I(\matrices_k)$ is a homogeneous ideal in the polynomial ring $\mf[\matrices]$ (with the standard grading), and the action 
of $GL_n$ preserves the grading, so $I(\matrices_k)^{GL_n}$ is a graded subspace; we denote by $I(\matrices_k)^{GL_n}_d$ the degree $d$ homogeneous component.

\begin{theorem}\label{thm:GL_n-inv} 
The space $I(\matrices_k)^{GL_n}_{(n-k)(n-k-1)}$ is spanned by $\sdisc_k$, and 
there are no $GL_n$-invariants of degree less than  $(n-k)(n-k-1)$ vanishing on $\matrices_k$. 
\end{theorem}

\begin{proof} Identify $\diag$ from Section~\ref{sec:invar} with the subspace of diagonal matrices in $\matrices$.  
Restriction of polynomial functions from $\matrices$ to $\diag$   gives an isomorphism of the graded algebras 
$\mf[\matrices]^{GL_n}\to \mf[\diag]^{S_n}$ by (a very special case of) the Chevalley restriction theorem. 
Moreover, since $\matrices_k$ contains $\diag_k$, the ideal $I(\matrices_k)$ is mapped into $I(\diag_k)$. 
So we have an injection 
$I(\matrices_k)^{GL_n}\to I(\diag_k)^{S_n}$ of graded vector spaces (in fact it is an isomorphism), and  the statement immediately follows from Lemma~\ref{lemma:S_n}. 
\end{proof} 


\section{Real symmetric matrices} \label{sec:realsymm} 

In this section we turn to the space $\rsm$  of $n\times n$ real symmetric matrices ($n\geq 2$). It is a vector space of dimension 
$n(n+1)/2$ over $\mr$. 
It contains the subset $\dsm_k$ of real symmetric matrices with at most $n-k-1$ distinct eigenvalues. Note that 
$\dsm_k$  is  a real algebraic subvariety of $\rsm$, as it is the zero locus of the polynomial function 
$\sdisc_k\in\mathbb{R}[\rsm]$ mapping $A\in\rsm$ to $\sdisc_k(A)$. 
The real orthogonal group $O_n$ acts on $\rsm$ by conjugation, and we consider the induced action of $O_n$ on $\mr[\rsm]$. 
The group $O_n$ preserves the subset $\dsm_k$, 
hence preserves also the vanishing  ideal $I(\dsm_k)$ of $\dsm_k$ in $\mathbb{R}[\rsm]$. Moreover, we shall write  
$I(\dsm_k)^{SO_n}_d$ for the degree $d$ homogeneous component of the space of $SO_n$-invariant polynomials in $I(\dsm_k)$, 
where $SO_n$ denotes the special orthogonal group over $\mr$. 

\begin{theorem}\label{thm:SO_n-inv} 
The space $I(\dsm_k)^{SO_n}_{(n-k)(n-k-1)}$ is spanned by $\sdisc_k$, 
and there are no $SO_n$-invariants of degree less than  $(n-k)(n-k-1)$ vanishing on 
$\dsm_k$.  
\end{theorem}

\begin{proof} Note first that $\sdisc$ is indeed an $SO_n$-invariant polynomial function on $\rsm$ vanishing on $\dsm_k$, and having degree $(n-k)(n-k-1)$. 
Each $SO_n$-orbit in $\rsm$ intersects the subspace $\diag$ of diagonal matrices in $\rsm$. 
Moreover, identifying $\diag$ and $\mathbb{R}^n$ in the obvious way, 
we have $\dsm_k\cap\diag=\diag_k$, and two diagonal matrices  belong to the same $SO_n$-orbit in $\rsm$ if and only if they belong to the same $S_n$-orbit in $\diag$. It follows that restriction of functions from $\rsm$ to $\diag$ gives a 
degree preserving  injection $I(\dsm_k)^{SO_n}_d\to I(\diag_k)^{S_n}_d$ (it is in fact an isomorphism). 
Thus our statements follow from Lemma~\ref{lemma:S_n}. 
\end{proof}


\section{On the minimal degree component of  $I(\dsm_k)$}\label{sec:irreducible}

Take  the $O_n$-module direct sum decomposition $\rsm=\tzrsm\oplus\mathbb{R}I$, where 
$\tzrsm$ denotes the subspace of trace zero symmetric matrices, and  $I$ is the $n\times n$ identity matrix.  
Projection from $\rsm$ to $\tzrsm$ gives an embedding of $\mathbb{R}[\tzrsm]$ as a subalgebra of $\mathbb{R}[\rsm]$, 
and $\sdisc_k$ belongs to $\mathbb{R}[\tzrsm]$.  
We extend a construction from \cite{domokos} from the special case $k=0$ to any 
$k\in\{0,1,\dots,n-2\}$.  Define a map 
\[\cov_k:\rsm \to \bigwedge^{n-k-1}{\mathcal{N}} \]  
(to the degree $n-k-1$ exterior power of $\tzrsm$) by 
\begin{equation}\label{eq:T}
\cov_k(A):=  \bigwedge_{i=1}^{n-k-1}  (A^i-\frac{1}{n}{\mathrm{Tr}}(A^i)I)
\end{equation}  
Clearly $\cov_k$ is an $O_n$-equivariant  map, and using that each $O_n$-orbit in $\rsm$ intersects the subspace of diagonal matrices in $\rsm$,  
 a straightforward extension of the proof of Proposition 4.1 in \cite{domokos} yields the following: 

\begin{proposition}\label{prop:cov} 
The matrix  $A\in\rsm$  belongs to $\dsm_k$  if and only if $\cov_k(A)=0$. 
\end{proposition}

Obviously $\cov_k$ is a polynomial map of degree $(n-k)(n-k-1)/2$, hence it makes sense to speak about its comorphism (in the sense of affine algebraic geometry)
$\cov_k^\star:\mathbb{R}[\bigwedge^{n-k-1}\tzrsm]\to\mathbb{R}[\rsm]$. We shall restrict the $\mr$-algebra homomorphism 
$\cov_k^\star$ to the linear component $(\bigwedge^{n-k-1}\tzrsm)^\star$ of the coordinate ring of 
$\bigwedge^{n-k-1}\tzrsm$, and by Proposition~\ref{prop:cov} we conclude the following:  

\begin{proposition}\label{prop:comorphism}  
$\cov_k^\star$ is a non-zero $O_n$-module map of the dual of  $\bigwedge^{n-1-k}\tzrsm$ into 
$I(\dsm_k)_{(n-k)(n-k-1)/2}$, such that  $\dsm_k$ is the common zero locus in $\rsm$ of the image under $\cov_k^\star$ of $(\bigwedge^{n-k-1}\tzrsm)^\star$.  
\end{proposition} 

\begin{corollary}\label{cor:I(E_k)} 
The minimal degree of a non-zero homogeneous polynomial function on $\rsm$ that vanishes on   $\dsm_k$
is $\frac{(n-k)(n-k-1)}{2}$. 
\end{corollary}

\begin{proof} $I(\dsm_k)_{(n-k)(n-k-1)/2}$ is non-zero by Proposition~\ref{prop:comorphism}. For  the reverse inequality note that since any $SO_n$-orbit in $\rsm$ intersects the subspace of diagonal matrices, a non-zero $SO_n$-invariant subspace 
in $I(\rsm_k)$ restricts to a non-zero $S_n$-submodule in $I(\diag_k)$, so the desired inequality follows directly from the Kleitman-Lov\'asz theorem \cite{lovasz}.
(Alternatively, a non-zero $SO_n$-invariant subspace of $I(\dsm_k)_l$ yields a non-zero $SO_n$-invariant in 
$I(\dsm_k)^{SO_n}_{2l}$ by Lemma 2.1 in \cite{domokos}, hence the second inequality follows from 
Theorem~\ref{thm:SO_n-inv}.) 
\end{proof} 

\begin{remark} 
The analogues of Corollary~\ref{cor:I(E_k)} does  not hold in the setup of Section~\ref{sec:conj}: for example, for $k=0$ we have that $\matrices_0$ is a $GL_n$-stable hypersurface 
in $\matrices$, so its vanishing ideal is generated by an $SL_n$-invariant on $\matrices$. Since scalar matrices act trivially on $\matrices$, we have $\mf[\matrices]^{SL_n}=\mf[\matrices]^{GL_n}$, 
so by Theorem~\ref{thm:GL_n-inv} $I(\matrices_0)$ is generated by $\disc$, and consequently 
$I(\matrices_0)_d=\{0\}$ for $d<n(n-1)$. 
\end{remark}


\section{A general upper bound for $\mu_k(n)$}\label{sec:general}

The construction of Section~\ref{sec:irreducible} can be used to get a bound on $\mu_k(n)$ by the following lemma: 

\begin{lemma}\label{lemma:mukn} 
Any non-zero $SO_ n$-invariant submodule $W$ in $I(\dsm_k)_{(n-k)(n-k-1)/2}$ (which is non-zero by Corollary~\ref{prop:comorphism}) has a basis $f_1,f_2,\dots$ such that 
$\sdisc_k=\sum f_i^2$. 
In particular, $\mu_k(n)$ is bounded from above by the minimal dimension of a non-zero $SO_n$-invariant subspace in the image of $(\bigwedge^{n-k-1}\tzrsm)^\star$ under the map $\cov_k^\star$. 
\end{lemma}  

\begin{proof} Any finite dimensional $SO_n$-submodule of $\mr[\rsm]$ has a basis $h_1,h_2,\dots $ such that $\sum h_i^2$ is $SO_n$-invariant, see Lemma 2.1 in 
\cite{domokos}. Apply this to $W$; by Theorem~\ref{thm:SO_n-inv}, $\sum h_i^2=C\cdot \sdisc_k$ for some scalar $C$. By positivity of the form we have $C>0$, and 
$f_i:=\frac{1}{\sqrt C}h_i$ is the desired basis of $W$. \end{proof}

Let $n\ge 2$ be a positive integer.  We have $n=2l$ or $n=2l+1$ for a positive integer $l$. 
Take a non-negative integer $k\le n-2$. Next we generalize Theorem 6.2 from \cite{domokos}, which is the special case $k=0$ of the statement below. 
A summary of the necessary background on representations of the orthogonal group  is given in Section 5   of 
\cite{domokos} (standard references for this material are \cite{weyl}, \cite{goodman-wallach}, \cite{fulton-harris}, \cite{procesi}). 
Recall that the finite dimensional irreducible complex representations of a connected compact Lie group are labeled by their highest weight. 
In case of $SO_n$ the highest weights (i.e. dominant integral weights) are usually identified with the $l$-tuples $\lambda=(\lambda_1,\dots,\lambda_l)$ of integers where 
\[\begin{cases} \lambda_1\ge\dots\ge\lambda_l\ge 0, &\mbox{ when }n=2l+1\\
\lambda_1\ge\dots\ge\lambda_{l-1}\ge |\lambda_l|,&\mbox{ when }n=2l. \end{cases} \]
Denote by $W_{\lambda}^{\mc}$ the irreducible complex $SO_n$-module with highest weight $\lambda$. Except when $n\equiv 2$ modulo $4$ and $\lambda_l\neq 0$, 
there is an irreducible real $SO_n$-module $W_{\lambda}$ whose complexification $\mc\otimes_{\mr}W_{\lambda}$ is $W_{\lambda}^{\mc}$. 
If $n\equiv 2$ modulo $4$ and $\lambda_l>0$, there is an irreducible real $O_n$-module $V_{\lambda}$ which remains irreducible as an $SO_n$-module, but its complexification splits as $\mc\otimes_{\mr}V_{\lambda}\cong W_{\lambda}^{\mc}+W_{(\lambda_1,\dots,\lambda_{l-1},-\lambda_l)}^{\mc}$ 
(see page 164 in \cite{weyl}). 
We shall use the notation $(a,1^r)$ for the sequence $(a,1,\dots,1)$ with $r$ copies of $1$.

\begin{theorem}\label{thm:main} 
The degree $(n-k)(n-k-1)/2$ homogeneous component of $I(\dsm_k)$ contains an irreducible 
$SO_n$-submodule isomorphic to 
\[V_{(n-l+1,1^{l-1})} \quad \mbox{ when }4 \mbox{ divides }n-2  \mbox{ and }k=l-1.\] 
Otherwise  $I(\dsm_k)_{(n-k)(n-k-1)/2}$  contains an irreducible $SO_n$-submodule isomorphic to 
\[\begin{cases}
W_{(n-k,1^k)} & \mbox{ when }k+1<n-l; \\ 
W_{(n-k, 1^{n-k-2})} & \mbox{ when }k+1\ge n-l . 
\end{cases}\]
\end{theorem} 

\begin{proof}  First we complexify and after a complex linear change of variables  we pass to the group 
$SO_n(\mc,J)$ preserving the quadratic form on $\mc^n$  with matrix 
$J:=\left(\begin{array}{cc}0& I \\  I & 0\end{array}\right)$ for $n=2l$ and 
$J:=\left(\begin{array}{ccc}0&I &0\\  I &0& 0\\ 0& 0& 1\end{array}\right)$ for $n=2l+1$, where $I$ is the $l\times l$ identity matrix. The space of symmetric matrices is replaced by the space $\rsm_{\mc,J}$ of  selfadjoint linear 
transformations of $(\mc^n,J)$, on which $SO_n(\mc,J)$ acts by conjugation, and $\tzrsm_{\mc,J}$ is the zero locus of the trace function on $\rsm_{\mc,J}$. 
 The map $\cov_k:\rsm_{\mc,J}\to\bigwedge^{n-k-1} \tzrsm_{\mc,J}$ is defined by the same formula as in  \eqref{eq:T}.  
Let $\mt$ be the maximal torus of $ SO_n(\mc,J)$ consisting of the diagonal matrices 
$\{{t=\mathrm{diag}}(t_1,\ldots,t_l,t_1^{-1},\ldots,t_l^{-1})\mid 
t_1,\ldots,t_l\in\mc^\times\}$ 
when $n=2l$ and 
$\{t={\mathrm{diag}}(t_1,\ldots,t_l,t_1^{-1},\ldots,t_l^{-1},1)\mid 
t_1,\ldots,t_l\in\mc^\times\}$ 
when $n=2l+1$. Weight vectors in an $SO_n(\mc,J)$-module are understood with respect to $\mt$. 
Denote $x_{ij}$ the function on $\tzrsm_{\mc,J}$ mapping an $n\times n$ matrix to its $(i,j)$-entry. 
They are weight vectors, in particular, we have 
\[t\cdot x_{i1}=\begin{cases} t_1t_i^{-1} &\mbox{ for }i=1,\dots,l;\\
t_1t_{i-l}&\mbox{ for }i=l+1,\dots,2l;\\
t_1&\mbox{ for }i=n=2l+1.\end{cases}\]
Order the functions $x_{i1}$ with $i>1$ into a sequence 
\begin{eqnarray*}
(x_1,\dots,x_{n-1}):=(x_{l+1,1},x_{l+2,1}, \dots,x_{n1}, 
x_{l1}, x_{l-1,1},  \dots, x_{21})
\end{eqnarray*} 
Note that the non-zero weights in $\rsm_{\mc,J}^\star$ are multiplicity free, 
the weights of the members of the above sequence strictly decrease with respect to the lexicographic ordering 
and they are all greater than the zero weight, 
and any weight of  $\rsm_{\mc,J}^\star$ not represented by the above sequence is strictly smaller with respect to the lexicographic ordering. 
It follows that for any $1\le s\le n-1$, 
the weight vector $x_1\wedge\dots\wedge x_s$ has maximal weight in $\bigwedge^s\tzrsm_{\mc,J}^\star$  
with respect to the lexicographic ordering (and the corresponding weight space is $1$-dimensional). 
Hence the weight of $x_1\wedge\dots\wedge x_s$ is maximal with respect to the natural partial ordering of weights explained in the last paragraph of Section 5 in \cite{domokos}. 
Consequently, $x_1\wedge\dots\wedge x_s$ is a highest weight vector in 
$\bigwedge^s\tzrsm_{\mc,J}^\star$, and one computes that its  weight is 
\[\begin{cases} (s+1,1^{s-1}) &\mbox{ for }s\le l; \\
(s+1, 1^{n-s-1})&\mbox{ for }n-1\ge s>l.
 \end{cases}\] 
Let $A$ denote the matrix of the linear transformation permuting the standard basis vectors  $e_1,\ldots,e_n\in\mc^n$ cyclically as follows: 
\[e_1\mapsto e_{l+1}\mapsto e_{l+2}\mapsto\cdots\mapsto e_n\mapsto e_l\mapsto e_{l-1}\mapsto 
\ldots\mapsto e_2\mapsto e_1.\]   
It is easy to see that $A$ belongs to $\tzrsm_{\mc,J}$. The first columns of the first $n$ powers of $A$  exhaust the set of standard basis vectors in $\mc^{n}$ in the order 
\[e_{l+1},e_{l+2},\dots,e_{n},e_{l},e_{l-1},\dots,e_2,e_1.\] 
Thus as explained in the proof of Proposition 6.1 in \cite{domokos},  the standard identification $\bigwedge^s\tzrsm^\star\cong (\bigwedge^s\tzrsm)^\star$ yields that 
$\cov_{n-s-1}^\star(x_1\wedge\dots\wedge x_s)(A)$ equals (up to sign) the determinant of the $s\times s$  identity matrix, hence is non-zero. We conclude that $\cov_k^\star(x_1\wedge\dots\wedge x_{n-k-1})$ is a non-zero highest weight vector in $\mc[\rsm_{\mc,J}]$ with the weight given above, and by Corollary~\ref{prop:comorphism}  it belongs to the  $\mc$-span of $I(\dsm_k)_{(n-k)(n-k-1)/2}$.  
It follows that the complexification of $I(\dsm_k)_{(n-k)(n-k-1)/2}$ contains as a summand the irreducible complex $SO_n$-module $W_{\lambda}^{\mc}$ where 
\[\lambda=\begin{cases} (n-k,1^k) & \mbox{ when }k+1<n-l; \\ 
(n-k, 1^{n-k-2}) & \mbox{ when }k+1\ge n-l . 
\end{cases}\] 
Hence our statement follows by the preceding discussion about irreducible real $SO_n$-modules and their complexifications. 
 \end{proof} 

Theorem~\ref{thm:main} has the following immediate consequence by Lemma~\ref{lemma:mukn}: 

\begin{corollary}\label{cor:bound} 
If  $4$ divides $n-2$ and $k=l-1$, then we have 
\[\mu_k(n)\le \dim_{\mr}(V_{(n-l+1,1^{l-1})}).\] 
Otherwise we have   
\[\mu_k(n)\le \begin{cases} \dim_{\mr}(W_{(n-k,1^k)}) & \mbox{ when }k+1<n-l; \\ 
\dim_{\mr}(W_{(n-k, 1^{n-k-2})}) & \mbox{ when }k+1\ge n-l . 
\end{cases}\]
\end{corollary} 

Of course $\dim_{\mr}(W_{\lambda})=\dim_{\mc}(W_{\lambda}^{\mc})$ for all irreducible $SO_n$-modules $W_{\lambda}$. 
When $n\equiv 2$ modulo $4$ and $\lambda_l\neq 0$, the complexification of the irreducible $SO_n$-module $V_{\lambda}$ splits as the sum of the equidimensional 
irreducible complex $SO_n$-modules $W_{\lambda}^{\mc}$ and $W_{(\lambda_1,\dots,\lambda_{l-1},-\lambda_l)}^{\mc}$, 
hence $\dim_{\mr}(V_{\lambda})=2\dim_{\mc}(W_{\lambda}^{\mc})$ in this case.  
For convenience of the reader we recall the formula for the dimension of  the irreducible complex $SO_n$-modules $W_{\lambda}^{\mc}$ with highest weight $\lambda$. 
For $n=2l$  even ($l\ge 2$) we have 
\[\dim_{\mc}(W_{\lambda}^{\mc})=\prod_{1\le i<j\le l} \frac{\lambda_i-\lambda_j+j-i}{j-i}\cdot \frac{\lambda_i+\lambda_j+n-i-j}{n-i-j}\] 
and for $n=2l+1$ odd we have 
\[\dim_{\mc}(W_{\lambda}^{\mc})=\prod_{1\le i<j\le l} \frac{\lambda_i-\lambda_j+j-i}{j-i}\cdot 
\prod_{1\le i\le j\le l}\frac{\lambda_i+\lambda_j+n-i-j}{n-i-j}\] 
(see for example page 410 of 
\cite{fulton-harris} or page 304 in \cite{goodman-wallach}). 

\begin{remark} \label{remark:royformula} 
{\rm 
(i)  Theorem 4.48 in \cite{roy} provides an explicit presentation of $\sdisc_k$ as a sum of squares where the number of summands is the binomial coefficient $\binom{\frac{n(n+1)}{2}}{n-k}$. 
This is the dimension of the space $\bigwedge^{n-k}\rsm$, containing properly a summand isomorphic to $\bigwedge^{n-k-1}\tzrsm$, and the latter in general is far from being irreducible as an $O_n$-module. This shows that the bound for $\mu_k(n)$ given in Theorem~\ref{thm:main} is significantly better. 

(ii) In fact the $O_n$-module map $\cov^\star_k:(\bigwedge^{n-k+1}\tzrsm)^\star\to I(\dsm_k)_{(n-k)(n-k-1)/2}$ is not injective. 
Generalizing Proposition 2 in \cite{gorodski}, 
denote by 
\[\gamma:\bigwedge^{n-k-1}\tzrsm \to so_n\otimes \bigwedge^{n-k-3}\tzrsm\]
 the $O_n$-module map 
\[a_1\wedge\dots\wedge a_{n-k-1}\mapsto \sum_{1\le i<j\le n-k-1}(-1)^{i+j}[a_i,a_j]\otimes a_1\wedge\dots\wedge \hat{a}_i\wedge\dots\wedge \hat{a}_j\wedge\dots\wedge a_{n-k-1}\] 
(here $so_n$ is the Lie algebra of $O_n$ endowed with the adjoint action).  
By definition of $\cov_k$ and $\gamma$ we have 
$\gamma\circ \cov_k=0$ (essentially because powers of a matrix commute), hence the image of the dual map 
$\gamma^\star:(so_n\otimes\bigwedge^{n-k-3}\tzrsm)^\star  \to (\bigwedge^{n-k-1}\tzrsm)^\star$ is contained in the 
kernel of $\cov_k^\star$. }
\end{remark}


\section{The $1$-subdiscriminant of $4\times 4$ matrices}\label{sec:4x4} 

As an example we treat the $1$-subdiscriminant of $4\times 4$ real symmmetric matrices in more detail. An easy standard calculation yields  the $SO_4$-module decomposition 
\begin{equation}\label{eq:so4decomp}
\bigwedge^2\tzrsm\cong W_{(3,1)} + W_{(3,-1)}+ W_{(1,1)}+W_{(1,-1)}.
\end{equation} 
The dimensions of the summands are $15,15,3,3$. Projection onto the sum of the last two summands can be identified with the surjection 
$\gamma:\bigwedge^2\tzrsm\to \mathrm{so}_4$ given by 
$\gamma(A\wedge B)=[A,B]=AB-BA$. 
Since powers of a matrix commute, by definition of $\gamma$ and $\cov_1$ we have $\gamma\circ \cov_1=0$ (compare with Remark~\ref{remark:royformula} (ii)).  
On the other hand the sum of the first two summands in \eqref{eq:so4decomp} 
is an irreducible $O_4$-submodule. 
We conclude by Proposition~\ref{prop:comorphism}  that 
\[\cov_1^\star((\bigwedge^2\tzrsm)^\star)\cong W_{(3,1)}+W_{(3,-1)}\] 
as $SO_4$-modules and hence by Lemma~\ref{lemma:mukn} we have 
\[\mu_1(4)\le \dim_{\mr}(W_{(3,1)})=15.\] 
(Note  that the formula of Theorem 4.48 in \cite{roy} yields an expression for the $1$-subdiscriminant of $4\times 4$ symmetric matrices  as a sum of $\binom{10}{3}=120$ squares.)


\end{document}